\documentclass[11pt]{article}

\usepackage{amsmath,amssymb,amsthm}
\usepackage{natbib}
\usepackage{geometry}
\usepackage{bm}
\usepackage{enumitem}
\usepackage{mathrsfs}
\usepackage{setspace}
\usepackage[colorlinks=true,linkcolor=blue,citecolor=blue,urlcolor=blue]{hyperref}

\usepackage{authblk}
\usepackage{orcidlink}

\newtheorem{theorem}{Theorem}
\newtheorem{proposition}{Proposition}

\newcommand{\Rr}{\mathbb R}
\newcommand{\calA}{\mathcal A}

\newcommand{\calZ}{\mathcal Z}
\newcommand{\calS}{\mathcal S}
\newcommand{\calP}{\mathcal P}

\newcommand{\ind}{\mathbf 1}

\title{A Leakage Bound for Confidence Sets after Black-Box Selection}

\author{Sayantan Banerjee \orcidlink{0000-0001-5414-4817}}
\affil{OM \& QT Area\\Indian Institute of Management Indore
}

\date{Email: \texttt{sayantanb@iimidr.ac.in}}

\begin{document}
	
	\maketitle
	
	\begin{abstract}
		In many analyses the object reported at the end is not fixed in advance, but is chosen after a preliminary search over variables, subgroups, transformations, models or contrasts. Classical selective-inference methods are most effective when this search can be written as an explicit selection event. This note treats the less structured case in which the selection rule is a black box and inference is required for the target indexed by the selected object. We show that, for any fixed-target confidence procedure, selected-target noncoverage is bounded by the nominal fixed-target noncoverage plus the average total variation distance between the marginal law of the inferential data and its conditional law given the selected object. A mutual-information bound follows immediately. The result recovers sample splitting as the zero-leakage case and gives explicit guarantees for noisy screening through a Gaussian information bound. Thus the inferential cost of black-box selection is quantified by the information that the selected object carries about the inferential sample.
	\end{abstract}
	
	
\section{Introduction}

Classical confidence statements are usually justified for a parameter, contrast or hypothesis specified before the data are examined. In many current analyses this order is reversed. A preliminary screening stage first selects variables, subgroups, outcomes, transformations, models or contrasts, and formal inference is then reported only for the selected object. The screening stage may be an algorithmic search, a machine-learning pipeline, proprietary software, or a human analyst interacting with the data. If the selection step is ignored, the resulting confidence intervals and tests need not have their nominal interpretation.

The effect of data-dependent selection on inference is well known. \citet{leeb2005model} showed that model selection can have a substantial impact on subsequent inference, even asymptotically. \citet{berk2013valid} proposed universally valid post-selection inference (PoSI) by protecting against all possible selected submodels. The PoSI approach protects against a large class of possible selected models by widening intervals; the present note instead keeps the fixed-target procedure and charges the dependence between the selected object and the inferential sample. A complementary line of work develops selective inference conditional on the event that a specified procedure has selected a given model or hypothesis; see, among others, \citet{fithian2014optimal}, \citet{lee2016exact}, \citet{taylor2015statistical} and \citet{tibshirani2016exact}. Randomization has also been used to smooth selection events and improve selective inference \citep{tian2018selective}. These approaches are most effective when the selection mechanism is mathematically available.

The present note is motivated by screening rules that are not available in a tractable mathematical form. The rule may be stochastic, proprietary, adaptive over several analyst interactions, or based on an external black-box system. In such cases it may be unrealistic to condition on the exact selection event, or even to write that event down. The question is then not how to adjust for a particular known selection rule, but what can be guaranteed when the rule is treated only through its output.

Let \(D\) denote the data used for final inference. A screening rule returns a selected object \(S^\star\) in a class \(\calS\). For each fixed \(s\in\calS\), let \(\theta_s(P)\) be the corresponding population target, and let \(C_s(D)\) be a confidence set computed from \(D\). The selected target is \(\theta_{S^\star}(P)\). Thus, after subgroup screening, the target is the population quantity for the selected subgroup; after variable screening, it may be the population projection parameter in the selected model. The target is random through the selected object.

The main result is a finite-sample coverage inequality. Suppose that, before screening, the confidence procedure has noncoverage probability at most \(\alpha\) for every fixed object that could be selected. After screening, the target itself is random because it is indexed by the selected object. We prove that
\[
\text{selected-target noncoverage}
\leq
\alpha+\text{information leakage}.
\]
The leakage term is the total variation distance, averaged over the screening output, between the original law of the inferential data and its conditional law after the selected object is known. A mutual-information bound gives a simpler, coarser version. Thus sample splitting is the zero-leakage case, while noisy, randomized or information-limited screening gives intermediate guarantees.

A related line of work studies validity under adaptive data analysis. \citet{dwork2015reusable} showed how a holdout sample can be reused while controlling overfitting, using ideas from algorithmic stability and privacy. \citet{bassily2016algorithmic} developed stability-based guarantees for adaptive queries, while \citet{russo2016controlling} and \citet{xu2017information} gave mutual-information bounds for bias and generalization error. These papers show that the information an adaptive procedure extracts from the data is a useful measure of its inferential cost. The distinction here is the object being controlled. Russo--Zou-type results control the bias of an adaptively chosen statistic, whereas the present result controls the probability of a selected-target noncoverage event. This event depends simultaneously on the selected object, the target attached to that object, and the confidence set reported after selection.

The contribution of the paper is a finite-sample selected-target coverage bound for settings in which the selection rule is unavailable as a mathematical object. The bound identifies the exact quantity that selection adds to the usual fixed-target error: the change in the law of the inferential data after the selected object is known. This formulation is useful because it does not require the geometry or probability of a selection event. It only requires control of the joint law of the selected object and the data used for inference. In this sense, sample splitting, noisy screening and compressed screening are different ways of controlling the same leakage term.

This perspective gives a design rule for black-box analyses. If the selected object is formed from data independent of the inferential sample, the leakage term vanishes. If the selected object is formed from a noisy or compressed summary, the leakage can be bounded in terms of the information carried by that summary. If the same inferential data are reused without restriction, the bound gives no nominal protection. The result therefore complements selective-inference methods that condition on a known selection event, while covering cases in which that event is not available in a tractable form.

The rest of the paper is organized as follows. Section~\ref{sec:setup} introduces selected targets and leakage. Section~\ref{sec:leakage} proves the main leakage bound and a finite sharpness example. Section~\ref{sec:noisy} gives a noisy-screening consequence and points to a finite-message variant in Appendix~\ref{app:compressed}. Section~\ref{sec:illustration} gives a small numerical illustration. Section~\ref{sec:discussion} discusses interpretation and limitations. Proofs are given in Appendix~\ref{app:proofs}.

\section{Selected targets and leakage}
\label{sec:setup}

Let \(D\) denote the data used to construct the final confidence set. In the simplest case, $D=(Z_1,\ldots,Z_m),$ where \(Z_1,\ldots,Z_m\) are independent observations from an unknown distribution \(P\in\calP\) on a measurable space \(\calZ\). The results below use only the law of \(D\) under \(P\), and therefore also apply when \(D\) is one fold or one
holdout part of a larger dataset.

A screening rule returns a random element $S^\star\in\calS,$ where \(\calS\) is a measurable selection space. An element \(s\in\calS\) may
represent a selected subgroup, model, variable set, transformation, contrast, outcome or hypothesis. We do not assume that the rule producing \(S^\star\) is available in analytic form. The theory uses only the joint law of \((S^\star,D)\) under \(P\).

For each fixed \(s\in\calS\), let $\theta_s(P)\in\Rr$ be the population target attached to \(s\). After screening, the inferential target is $
\theta_{S^\star}(P).$ Thus the target is allowed to be random through the selected object. This is the selected-target formulation: after subgroup screening, the target is the population quantity for the selected subgroup; after model screening, it is the parameter associated with the selected model.

For each fixed \(s\), let \(C_s(D)\) be a confidence set for \(\theta_s(P)\). The fixed-target noncoverage requirement is
\begin{equation}
	\label{eq:fixed-noncoverage}
	\mathrm{pr}_P\{\theta_s(P)\notin C_s(D)\}\leq \alpha,
	\qquad s\in\calS .
\end{equation}

This is a pointwise fixed-target condition. It does not assert simultaneous coverage over all \(s\in\mathcal S\). If several selected targets are reported together, the fixed-target procedure must first provide the corresponding simultaneous guarantee.

To measure this dependence, write $\mu_P=\mathcal L_P(D)$ for the marginal law of the inferential data, and let $\mathcal L_P(D\mid S^\star)$ denote a regular conditional law. The total-variation leakage is
\begin{equation}
	\label{eq:tv-leakage}
	\Lambda_{\rm TV}(P)
	=
	E_P\left[
	d_{\rm TV}\{\mathcal L_P(D\mid S^\star),\mathcal L_P(D)\}
	\right],
\end{equation}
where
\[
d_{\rm TV}(\nu,\mu)=\sup_A|\nu(A)-\mu(A)|.
\]
The expectation in \eqref{eq:tv-leakage} is over the random selected object \(S^\star\). Thus \(\Lambda_{\rm TV}(P)=0\) when \(S^\star\) is independent of \(D\). More generally, it measures how much the distribution of the inferential sample changes once the selected object is known.

We shall also use the mutual information
\[
I_P(S^\star;D)
=
{\rm KL}\!\left\{
\mathcal L_P(S^\star,D)
\,\middle\|\,
\mathcal L_P(S^\star)\otimes\mathcal L_P(D)
\right\},
\]
with the convention that it is \(+\infty\) if the joint law is not absolutely continuous with respect to the product law. By Pinsker's inequality and Jensen's inequality,
\begin{equation}
	\label{eq:tv-mi}
	\Lambda_{\rm TV}(P)
	\leq
	\{I_P(S^\star;D)/2\}^{1/2}.
\end{equation}
The next section bounds the additional selected-target noncoverage by
\(\Lambda_{\rm TV}(P)\), and hence by the mutual-information quantity in
\eqref{eq:tv-mi}.	

\section{Main bound and sharpness}
\label{sec:leakage}

We now state the main result. The theorem is a selected-target coverage bound: it starts from fixed-target validity and adds a penalty for the information carried by the selected object about the inferential sample. 
Throughout this section, assume that \(\calS\) is a standard Borel space. We also assume that the noncoverage set $\{(s,d):\theta_s(P)\notin C_s(d)\} \subseteq \calS\times\mathcal D$
is measurable for each \(P\in\calP\), where \(\mathcal D\) denotes the sample space of \(D\). A simple sufficient condition is the following: \(s\mapsto \theta_s(P)\) is measurable, and \(C_s(d)\) is an interval
$ C_s(d)=[L(s,d),U(s,d)]$ with \(L\) and \(U\) jointly measurable in \((s,d)\). Then
\[
\{(s,d):\theta_s(P)\notin C_s(d)\}
=
\{(s,d):\theta_s(P)<L(s,d)\}
\cup
\{(s,d):\theta_s(P)>U(s,d)\},
\]
which is measurable. This covers the usual confidence intervals considered below. For finite or countable \(\calS\), the condition is automatic once \(C_s(D)\) is measurable for each fixed \(s\). These are standard product-measurability conditions; see, for example,
\citet{kallenberg2002foundations}.

\begin{theorem}[Leakage bound]
	\label{thm:tv-leakage}
	Suppose that the fixed-target noncoverage bound
	\[
	\mathrm{pr}_P\{\theta_s(P)\notin C_s(D)\}\leq \alpha,
	\qquad s\in\calS,
	\]
	holds for a given \(P\in\calP\). Then
	\[
	\mathrm{pr}_P\{\theta_{S^\star}(P)\notin C_{S^\star}(D)\}
	\leq
	\alpha+\Lambda_{\rm TV}(P).
	\]
	Consequently,
	\[
	\mathrm{pr}_P\{\theta_{S^\star}(P)\notin C_{S^\star}(D)\}
	\leq
	\alpha+\{I_P(S^\star;D)/2\}^{1/2}.
	\]
\end{theorem}

\begin{proof}
	Let
	\[
	E_s=\{d:\theta_s(P)\notin C_s(d)\}.
	\]
	By the measurability assumption, \((s,d)\mapsto\ind\{d\in E_s\}\) is measurable. Let \(K_P(s,\cdot)=\mathcal L_P(D\mid S^\star=s)\) be a regular conditional law	of \(D\) given \(S^\star=s\), and let \(\mu_P=\mathcal L_P(D)\). Then
	\[
	\begin{aligned}
		\mathrm{pr}_P\{\theta_{S^\star}(P)\notin C_{S^\star}(D)\}
		&=
		E_P\{K_P(S^\star,E_{S^\star})\}.
	\end{aligned}
	\]
	For each \(s\),
	\[
	K_P(s,E_s)
	\leq
	\mu_P(E_s)+d_{\rm TV}\{K_P(s,\cdot),\mu_P\}.
	\]
	The fixed-target assumption gives \(\mu_P(E_s)\leq\alpha\). Hence
	\[
	\begin{aligned}
		\mathrm{pr}_P\{\theta_{S^\star}(P)\notin C_{S^\star}(D)\}
		&\leq
		\alpha+
		E_P\left[
		d_{\rm TV}\{K_P(S^\star,\cdot),\mu_P\}
		\right]  \\
		&=
		\alpha+\Lambda_{\rm TV}(P).
	\end{aligned}
	\]
	The mutual-information bound follows from \eqref{eq:tv-mi}.
\end{proof}

The theorem isolates the role of selection: it affects coverage only through the change from \(\mathcal L_P(D)\) to \(\mathcal L_P(D\mid S^\star)\). The zero-leakage case gives ordinary split-sample validity; positive but small leakage gives the approximate form.

The following construction shows sharpness in the unrestricted class considered by Theorem~\ref{thm:tv-leakage}. It should not be read as a model for ordinary parametric confidence intervals. The confidence sets are chosen to align their noncoverage events with the parts of the sample space most favoured by the conditional laws. Its role is only to show that, without additional structure on the procedure or on the selector, the coefficient of the leakage term cannot be reduced. 

\begin{proposition}[Sharpness]
	\label{prop:sharpness}
	Let \(0<\alpha\leq 1/2\) and \(0\leq\delta\leq\alpha\). Let
	\(\mathcal D=\{a_0,a_1,b\}\). Define two probability measures on \(\mathcal D\)
	by
	\[
	\begin{array}{lll}
		\mu_0(a_0)=\alpha+\delta, &
		\mu_0(a_1)=\alpha-\delta, &
		\mu_0(b)=1-2\alpha,\\[2mm]
		\mu_1(a_0)=\alpha-\delta, &
		\mu_1(a_1)=\alpha+\delta, &
		\mu_1(b)=1-2\alpha .
	\end{array}
	\]
	Let \(\mathrm{pr}(S^\star=0)=\mathrm{pr}(S^\star=1)=1/2\), and let
	$
	\mathcal L(D\mid S^\star=j)=\mu_j,\;j=0,1.
	$
	Then the marginal law \(\mu=\mathcal L(D)\) satisfies
	\[
	\mu(a_0)=\mu(a_1)=\alpha,\qquad \mu(b)=1-2\alpha,
	\]
	and
	\[
	d_{\rm TV}(\mu_0,\mu)=d_{\rm TV}(\mu_1,\mu)=\delta.
	\]
	There exist targets \(\theta_0,\theta_1\) and confidence sets \(C_0(D),C_1(D)\)
	such that
	\[
	\mathrm{pr}_{\mu}\{\theta_j\notin C_j(D)\}=\alpha,\qquad j=0,1,
	\]
	but
	\[
	\mathrm{pr}\{\theta_{S^\star}\notin C_{S^\star}(D)\}
	=
	\alpha+
	\frac12 d_{\rm TV}(\mu_0,\mu)
	+
	\frac12 d_{\rm TV}(\mu_1,\mu).
	\]
	Thus the bound in Theorem~\ref{thm:tv-leakage} is attained in the unrestricted
	class of confidence procedures and screening rules.
\end{proposition}

We provide the proof of the above proposition in Appendix~\ref{app:proofs}.

\section{Noisy screening and calibrated inference}
\label{sec:noisy}

The leakage bound becomes operational when the dependence between the selected object and the inferential data can be bounded by design. One simple route is to allow the screening rule to see only a noisy summary of the inferential sample.

Suppose the screening rule does not observe \(D\) directly. Instead, it observes $W=T(D)+\xi,$ where \(T(D)\in\Rr^q\), \(\xi\sim N(0,\tau^2 I_q)\), and \(\xi\) is independent of \(D\). The selected object is
$S^\star=\calA(W,U),$ where \(U\) is any additional randomization independent of \((D,\xi)\). The map \(\calA\) may be arbitrary. The entropy calculation is the standard covariance-constrained Gaussian channel bound; see, for example, \citet[Chapter~9]{cover2006elements}.

\begin{proposition}[Gaussian noisy screening]
	\label{prop:gaussian-noisy-screening}
	Assume that \(\operatorname{cov}_P\{T(D)\}\) exists, and write
	\[
	\Sigma_T(P)=\operatorname{cov}_P\{T(D)\}.
	\]
	If the fixed-target confidence sets satisfy
	\[
	\mathrm{pr}_P\{\theta_s(P)\notin C_s(D)\}\leq \alpha,
	\qquad s\in\calS ,
	\]
	then
	\[
	\mathrm{pr}_P\{\theta_{S^\star}(P)\notin C_{S^\star}(D)\}
	\leq
	\alpha+
	\left[
	\frac14
	\log\det\{I_q+\tau^{-2}\Sigma_T(P)\}
	\right]^{1/2}.
	\]
	In particular, if \(\operatorname{tr}\{\Sigma_T(P)\}\leq v\), then
	\[
	\mathrm{pr}_P\{\theta_{S^\star}(P)\notin C_{S^\star}(D)\}
	\leq
	\alpha+
	\left[
	\frac{q}{4}
	\log\left\{1+\frac{v}{q\tau^2}\right\}
	\right]^{1/2}.
	\]
\end{proposition}

We provide the proof of the above proposition in the Appendix. In applications, \(\Sigma_T(P)\) may be bounded by design. For example, if the components of \(T(D)\) are known to lie in intervals of lengths \(B_1,\ldots,B_q\), then \(\operatorname{Var}\{T_j(D)\}\leq B_j^2/4\), giving
\[
\operatorname{tr}\{\Sigma_T(P)\}
\leq
\frac14\sum_{j=1}^q B_j^2 .
\]
Alternatively, a conservative pilot or design-based upper bound on
\(\operatorname{tr}\{\Sigma_T(P)\}\) may be used to choose the noise level \(\tau\). The bound is intended as a design-based calibration principle, not as an after-the-fact estimate from a single dataset. If the screening summary is low-dimensional, or if its covariance is small relative to the added noise, then the coverage loss is controlled. The downstream screening map \(\calA\) does not enter the bound; all information about \(D\) available to the selector must pass
through \(W\). A second route is to restrict the screening rule to a finite or compressed summary of the inferential data. Appendix~\ref{app:compressed} records a simple finite-message bound that illustrates another way in which leakage can be controlled by design.

Sample splitting is the limiting case in which the selector receives no
information from \(D\). If \(S^\star\) is computed from an independent screening sample, then \(I_P(S^\star;D)=0\), and Theorem~\ref{thm:tv-leakage} gives
\[
\mathrm{pr}_P\{\theta_{S^\star}(P)\notin C_{S^\star}(D)\}\leq \alpha.
\]
More generally, for a sequence of problems, if
\[
\sup_{P\in\calP}\sup_{s\in\calS}
\mathrm{pr}_P\{\theta_s(P)\notin C_{s,m}(D_m)\}
\leq \alpha+r_m,
\qquad r_m\to0,
\]
and
\[
\sup_{P\in\calP}I_P(S_m^\star;D_m)\leq \eta_m,
\qquad \eta_m\to0,
\]
then
\[
\sup_{P\in\calP}
\mathrm{pr}_P\{\theta_{S_m^\star}(P)\notin C_{S_m^\star,m}(D_m)\}
\leq
\alpha+r_m+(\eta_m/2)^{1/2}.
\]
Thus ordinary fixed-target validity extends to selected-target validity whenever the selected object carries asymptotically negligible information about the inferential sample. 

\section{A small illustration}
\label{sec:illustration}

We give a small numerical illustration of the leakage principle. The purpose is not to compare methods, but to show the qualitative behaviour predicted by the bound: same-sample screening can undercover, independent splitting restores coverage, and noisy screening gives an intermediate regime.

For each replication, generate
\[
X_i=(X_{i1},\ldots,X_{ip})^\top,\qquad i=1,\ldots,n,
\]
with independent \(N(0,1)\) coordinates. We take \(n=400\) and \(p=50\). For \(j=1,\ldots,p\), define the population contrast
\[
\theta_j=E(X_{ij})=0.
\]
The screening rule selects one coordinate. The reported confidence interval is the usual normal interval for the mean of the selected coordinate,
\[
\widehat\theta_j\pm z_{\alpha/2}n^{-1/2},
\qquad \alpha=0.05.
\]
If \(j\) were fixed in advance, this interval has exact coverage \(1-\alpha\).

We compare four screening designs. In the same-sample design, the selected coordinate is
\[
S^\star=\arg\max_{1\leq j\leq p}\widehat\theta_j,
\]
where \(\widehat\theta_j=n^{-1}\sum_{i=1}^n X_{ij}\), and the interval is
computed from the same data. In the split-sample design, half the observations are used to select the coordinate and the other half are used to form the interval. In the noisy-screening design, the selector observes
\[
W_j=\widehat\theta_j+\xi_j,\qquad
\xi_j\sim N(0,\tau^2),
\]
independently over \(j\), and selects the coordinate with the largest \(W_j\); the confidence interval is still computed from the original data. We vary
\(\tau\).

Table~\ref{tab:illustration} reports empirical coverage over \(10,000\) replications. The same-sample selector substantially undercovers, because it preferentially chooses a coordinate with a large positive noise fluctuation. In this example, even \(\tau=0.25\) is large relative to the standard error \(1/\sqrt{400}=0.05\) of each coordinate mean; hence the noisy selector is already close to being decoupled from the inferential statistic. Coverage consequently returns rapidly
towards the nominal level as noise is added. 

\begin{table}[h]
	\centering
	\caption{Empirical coverage for the selected coordinate mean, based on 10,000 replications.}
	\label{tab:illustration}
	\begin{tabular}{lcc}
		\hline
		Screening design & Empirical coverage & Monte Carlo s.e. \\
		\hline
		No selection; fixed coordinate & 0.95 & 0.002 \\
		Same-sample selection & 0.28 & 0.004 \\
		Noisy screening, $\tau = 0.25$ & 0.93 & 0.003 \\
		Noisy screening, $\tau = 0.50$ & 0.95 & 0.002 \\
		Noisy screening, $\tau = 1.00$ & 0.95 & 0.002 \\
		Split-sample selection & 0.95 & 0.002 \\
		\hline
	\end{tabular}
\end{table}

The contrast between the first two nontrivial rows is the post-selection effect: using the same data to select the largest empirical mean and to form the confidence interval destroys coverage. Sample splitting removes this dependence. Noisy screening gives an intermediate design. When the added noise is large relative to the sampling variation in the screening statistics, the selected coordinate carries little information about the inferential fluctuation, and the coverage is close to the fixed-target level. This is consistent with the leakage bound and with Proposition~\ref{prop:gaussian-noisy-screening}.  For the same-sample selector, the coverage can be computed exactly. If \(M_p=\max_{1\leq j\leq p}\sqrt n\,\widehat\theta_j\), then \(M_p\) is the maximum of \(p\) independent standard normal variables. The selected interval covers zero exactly when \(M_p\leq z_{\alpha/2}\), so the coverage is \(
\Phi(z_{\alpha/2})^p.\) With \(p=50\) and \(\alpha=0.05\), this equals approximately \(0.28\), matching the simulation.

The information bound is not intended to be numerically sharp in this toy
example. The table illustrates the qualitative decoupling effect of noisy
screening: once the added noise is large relative to the sampling variability of the coordinate means, the selected coordinate carries much less information about the inferential fluctuation.

\section{Discussion}
\label{sec:discussion}

This paper considers post-selection coverage when the selection rule is not available as a tractable mathematical object. The main result shows that fixed-target coverage can be transferred to the selected target with an additional leakage term: the average total variation distance between the marginal law of the inferential data and its conditional law given the selected object. The statement is finite-sample and does not require the selection event itself to be characterized.

The main contribution is a reduction of black-box post-selection inference to a dependence problem. Selective inference usually exploits a known selection event, while PoSI-type inference protects against many possible selected models by enlarging intervals. Here the fixed-target procedure is kept unchanged, and the extra cost is measured through the dependence between the selected object and the inferential data. This is useful when the screening step is stochastic, proprietary, interactive, or otherwise difficult to express analytically.

In applications, the leakage term is best treated as a design quantity. It can be reduced by separating screening from inference, or by restricting what the screening step can learn through noise, compression, privacy, or stability constraints. The noisy-screening result and the finite-message example illustrate two such controls, but the same principle can be built into other workflows. The coverage statement should also be read narrowly. It concerns the target attached to the object actually reported; it does not validate the scientific importance or stability of that object. Nor does it supply fixed-target validity, which remains a separate issue in high-dimensional, semiparametric, or nuisance-estimation settings.

Several refinements remain open. The total-variation bound is deliberately general and may be conservative for structured procedures such as one-sided intervals, Wald intervals, likelihood-ratio regions, or conformal sets. Other dependence measures, including max-information, maximal leakage, differential privacy, or algorithmic stability, may give sharper bounds for particular screening mechanisms. A useful next step is to develop analysis protocols that choose the amount of randomization, compression, or splitting needed to keep the leakage below a prescribed tolerance.

\bibliographystyle{apalike}
\bibliography{ref.bib}

\appendix

\section{Proofs of results}
\label{app:proofs}

\begin{proof}[Proof of Proposition~\ref{prop:sharpness}]
	Take \(\theta_0=\theta_1=0\), and define
	\[
	C_0(D)=
	\begin{cases}
		\varnothing, & D=a_0,\\
		\{0\}, & D\neq a_0,
	\end{cases}
	\qquad
	C_1(D)=
	\begin{cases}
		\varnothing, & D=a_1,\\
		\{0\}, & D\neq a_1 .
	\end{cases}
	\]
	Under the marginal law \(\mu\),
	\[
	\mathrm{pr}_{\mu}\{\theta_0\notin C_0(D)\}=\mu(a_0)=\alpha,
	\qquad
	\mathrm{pr}_{\mu}\{\theta_1\notin C_1(D)\}=\mu(a_1)=\alpha.
	\]
	Thus the two fixed-target confidence sets have noncoverage exactly \(\alpha\).
	
	For \(\mu_0\), the signed differences from \(\mu\) are \(+\delta\) at \(a_0\),
	\(-\delta\) at \(a_1\), and \(0\) at \(b\). Hence
	\[
	d_{\rm TV}(\mu_0,\mu)
	=
	\sup_A|\mu_0(A)-\mu(A)|
	=
	\delta,
	\]
	with the supremum attained at \(A=\{a_0\}\). The same argument gives
	\(d_{\rm TV}(\mu_1,\mu)=\delta\).
	
	Moreover,
	\[
	\text{pr}\{\theta_{S^\star}\notin C_{S^\star}(D)\}
	=
	\frac12\mu_0(a_0)+\frac12\mu_1(a_1)
	=
	\frac12(\alpha+\delta)+\frac12(\alpha+\delta)
	=
	\alpha+\delta.
	\]
	
	Hence
	\[
	\alpha+\delta
	=
	\alpha+
	\frac12 d_{\rm TV}(\mu_0,\mu)
	+
	\frac12 d_{\rm TV}(\mu_1,\mu),
	\]
	which is the bound in Theorem~\ref{thm:tv-leakage}.
\end{proof}

\begin{proof}[Proof of Proposition~\ref{prop:gaussian-noisy-screening}]
	Since \(S^\star\) is a measurable function of \((W,U)\), with \(U\) independent of \(D\), the data-processing inequality gives
	\[
	I_P(S^\star;D)\leq I_P(W;D).
	\]
	Let \(T=T(D)\). Since \(T\) is a measurable function of \(D\), and
	\(W=T+\xi\) with \(\xi\) independent of \(D\), we have
	\(W\perp D\mid T\). Hence, by the chain rule for mutual information,
	\[
	I_P(W;D)
	=
	I_P(W;T)+I_P(W;D\mid T)
	=
	I_P(W;T).
	\]
	Moreover,
	\[
	I_P(W;T)
	=
	h_P(W)-h_P(W\mid T)
	=
	h_P(W)-h(\xi).
	\]
	The random vector \(W\) has covariance \(\operatorname{cov}_P(W)=\Sigma_T(P)+\tau^2 I_q.\)
	Since Gaussian distributions maximize differential entropy at fixed covariance,
	\[
	h_P(W)
	\leq
	\frac12\log\{(2\pi e)^q\det(\Sigma_T(P)+\tau^2 I_q)\}.
	\]
	Also
	\[
	h(\xi)
	=
	\frac12\log\{(2\pi e)^q\det(\tau^2 I_q)\}.
	\]
	Therefore
	\[
	I_P(W;D)
	\leq
	\frac12
	\log\det\{I_q+\tau^{-2}\Sigma_T(P)\}.
	\]
	Combining this with Theorem~\ref{thm:tv-leakage} and the mutual-information bound gives
	\[
	\mathrm{pr}_P\{\theta_{S^\star}(P)\notin C_{S^\star}(D)\}
	\leq
	\alpha+
	\left[
	\frac14
	\log\det\{I_q+\tau^{-2}\Sigma_T(P)\}
	\right]^{1/2}.
	\]
	
	For the trace bound, let \(\lambda_1,\ldots,\lambda_q\) be the eigenvalues of \(\Sigma_T(P)\). Then
	\[
	\log\det\{I_q+\tau^{-2}\Sigma_T(P)\}
	=
	\sum_{j=1}^q\log(1+\lambda_j/\tau^2).
	\]
	By concavity of \(x\mapsto\log(1+x)\),
	\[
	\sum_{j=1}^q\log(1+\lambda_j/\tau^2)
	\leq
	q\log\left\{
	1+\frac{1}{q\tau^2}\sum_{j=1}^q\lambda_j
	\right\}.
	\]
	If \(\operatorname{tr}\{\Sigma_T(P)\}\leq v\), the last display is bounded by
	\[
	q\log\left\{1+\frac{v}{q\tau^2}\right\}.
	\]
	This proves the stated trace bound.
\end{proof}

\section{A compressed-screening example}
\label{app:compressed}

This appendix records a simple finite-message version of the leakage principle. It is included to show that the bound is not tied to Gaussian perturbation. If a screening rule can only observe a coarse summary of the inferential data, then the selected object cannot carry more information about the inferential sample than that summary.

Let \(W=\varphi(D)\) be a measurable summary of the inferential data taking
values in a finite set \(\mathcal W\). The selected object is
\[
S^\star=\mathcal A(W,U),
\]
where \(U\) is additional randomization independent of \(D\). The map
\(\mathcal A\) may be arbitrary.

\begin{proposition}[Finite-message screening]
	\label{prop:finite-message}
	Suppose \(W=\varphi(D)\) takes values in a finite set \(\mathcal W\), and
	\(S^\star=\mathcal A(W,U)\), with \(U\) independent of \(D\). If the fixed-target
	confidence sets satisfy
	\[
	\operatorname{pr}_P\{\theta_s(P)\notin C_s(D)\}\leq \alpha,
	\qquad s\in\mathcal S,
	\]
	then
	\[
	\operatorname{pr}_P\{\theta_{S^\star}(P)\notin C_{S^\star}(D)\}
	\leq
	\alpha+
	\left\{\frac{1}{2}H_P(W)\right\}^{1/2},
	\]
	where \(H_P(W)\) is the entropy of \(W\), measured in nats. In particular,
	\[
	\operatorname{pr}_P\{\theta_{S^\star}(P)\notin C_{S^\star}(D)\}
	\leq
	\alpha+
	\left\{\frac{1}{2}\log |\mathcal W|\right\}^{1/2}.
	\]
\end{proposition}

\begin{proof}
	Since \(S^\star\) is a measurable function of \((W,U)\), with \(U\) independent
	of \(D\), the data-processing inequality gives
	\[
	I_P(S^\star;D)\leq I_P(W;D).
	\]
	Because \(W=\varphi(D)\) is a deterministic function of \(D\),
	\[
	I_P(W;D)=H_P(W).
	\]
	The mutual-information form of Theorem~\ref{thm:tv-leakage} therefore gives
	\[
	\operatorname{pr}_P\{\theta_{S^\star}(P)\notin C_{S^\star}(D)\}
	\leq
	\alpha+\{H_P(W)/2\}^{1/2}.
	\]
	Finally, \(H_P(W)\leq \log|\mathcal W|\), which proves the second bound.
\end{proof}

The bound is useful only when the summary alphabet is small enough for the entropy term to be nontrivial. Its role is mainly conceptual: it shows that leakage can be controlled not only by adding noise, as in
Proposition~\ref{prop:gaussian-noisy-screening}, but also by restricting the amount of information released to the screening rule. A sharper version can be obtained whenever a problem-specific entropy bound on \(W\) is available.

\end{document}